\documentclass[11pt]{article} 
\usepackage{overpic}
\usepackage{bm}
\usepackage{amsmath,amssymb,latexsym}
\usepackage{amsthm} 
\usepackage{overpic}
\usepackage{pxfonts} 

\newtheorem{theorem}{Theorem}[section] 
\newtheorem{lemma}[theorem]{Lemma}
\newtheorem{proposition}[theorem]{Proposition}

\theoremstyle{definition}

\theoremstyle{remark}

\numberwithin{equation}{section} 
\newcommand{\field}[1]{\mathbb{#1}} 
\newcommand{\R}{\field{R}}

\newcommand{\N}{\field{N}} 
\newcommand{\C}{\field{C}}

\newcommand{\res}{\mathop{\rm res}}

\newcommand{\supp}{\mathop{\rm supp}}

\newcommand{\const}{{\rm const}}

\renewcommand{\Re}{\mathop{\rm Re}}
\renewcommand{\Im}{\mathop{\rm Im}}

\title{Trajectories of quadratic differentials for Jacobi polynomials with complex parameters}

\author{Andrei Mart\'{\i}nez-Finkelshtein, Pedro Mart\'{\i}nez-Gonz\'alez, \\ and Faouzi Thabet}

\begin{document}

\vspace{1cm} \maketitle

\begin{abstract}
Motivated by the study of the asymptotic behavior of Jacobi polynomials $\left( P_{n}^{(nA,nB)}\right) _{n}$ 
with $A\in \C$ and $B>0$ we establish the global structure of trajectories of the related rational quadratic differential on $\C$. As a consequence, the asymptotic zero distribution (limit of the root-counting
measures of $\left( P_{n}^{(nA,nB)}\right) _{n}$) is described. The support of this measure is formed by an open arc in the complex plan (critical trajectory of the aforementioned quadratic differential) that can be characterized by the symmetry property of its equilibrium measure in a certain external field.
\end{abstract}



\section{Introduction} \label{sec:Intro}

The motivation of this work is the large-degree analysis 
of the behavior of the Jacobi polynomials $P_{n}^{(\alpha ,\beta) }$, when the parameters $\alpha$, $\beta$ are complex and  
depend on the degree $n$ linearly. Recall that these polynomials can be given explicitly by (see \cite{NIST,szego:1975}) 
\begin{equation*} 
P_{n}^{(\alpha ,\beta )}\left( z\right) =2^{-n}\sum_{k=0}^{n}\left( 
\begin{array}{c}
n+\alpha \\ 
n-k%
\end{array}%
\right) \left( 
\begin{array}{c}
n+\beta \\ 
\ k%
\end{array}%
\right) \left( z-1\right) ^{k}\left( z+1\right) ^{n-k},
\end{equation*}
or, equivalently, by the well-known Rodrigues formula
\begin{equation} \label{RodrJac}
P_{n}^{(\alpha ,\beta) }\left( z\right) =\frac{1}{2^{n}n!}\left( z-1\right)
^{-\alpha }\left( z+1\right) ^{-\beta }\left( \frac{d}{dz}\right) ^{n}\left[
\left( z-1\right) ^{n+\alpha }\left( z+1\right) ^{n+\beta }\right] .
\end{equation}
Clearly, polynomials $P_{n}^{(\alpha ,\beta) }$ are entire functions of the complex parameters $\alpha
,\beta  $.

If we fix $\alpha ,\beta  \in \C$  and allow $n\to \infty$, the zeros of $P_{n}^{(\alpha ,\beta )}$ will cluster on $[-1,1]$ and distribute there according to  the well-known arcsine law. A non-trivial asymptotic behavior can be obtained in the case of varying coefficients $\alpha$ and $\beta$. Namely, we will consider sequences 
\begin{equation}
\label{eq_P}
p_n(z)=P_{n}^{(\alpha_n ,\beta_n )}(z), \quad \alpha_n=nA, \quad \beta_n=nB,
\end{equation}
where both $A$ and $B$ are fixed. The case $A, B\geq 0$ can be studied by the already standard techniques from the potential theory \cite{Gonchar:84} or by the saddle point method applied to their integral representation, see e.g.~\cite{MR1114785}. The general situation $A, B\in \R$ was analyzed in \cite{MR2124460,MR1805976,MR2142296}. 

In this paper we are interested in the situation when at least one of the parameters, $A$ or $B$, is non-real, see e.g.~Figure~\ref{fig:onlyzeros}. To be more precise, we assume that  
\begin{equation}
A \notin  \R , \quad B>0.  
\label{cond sur Aet B}
\end{equation}
Clearly, results for the case $A>0$ and $B \notin  \R$ can be easily deduced by reversing the roles of $1$ and $-1$.

\begin{figure}[htb]
\centering \begin{overpic}[scale=0.7]{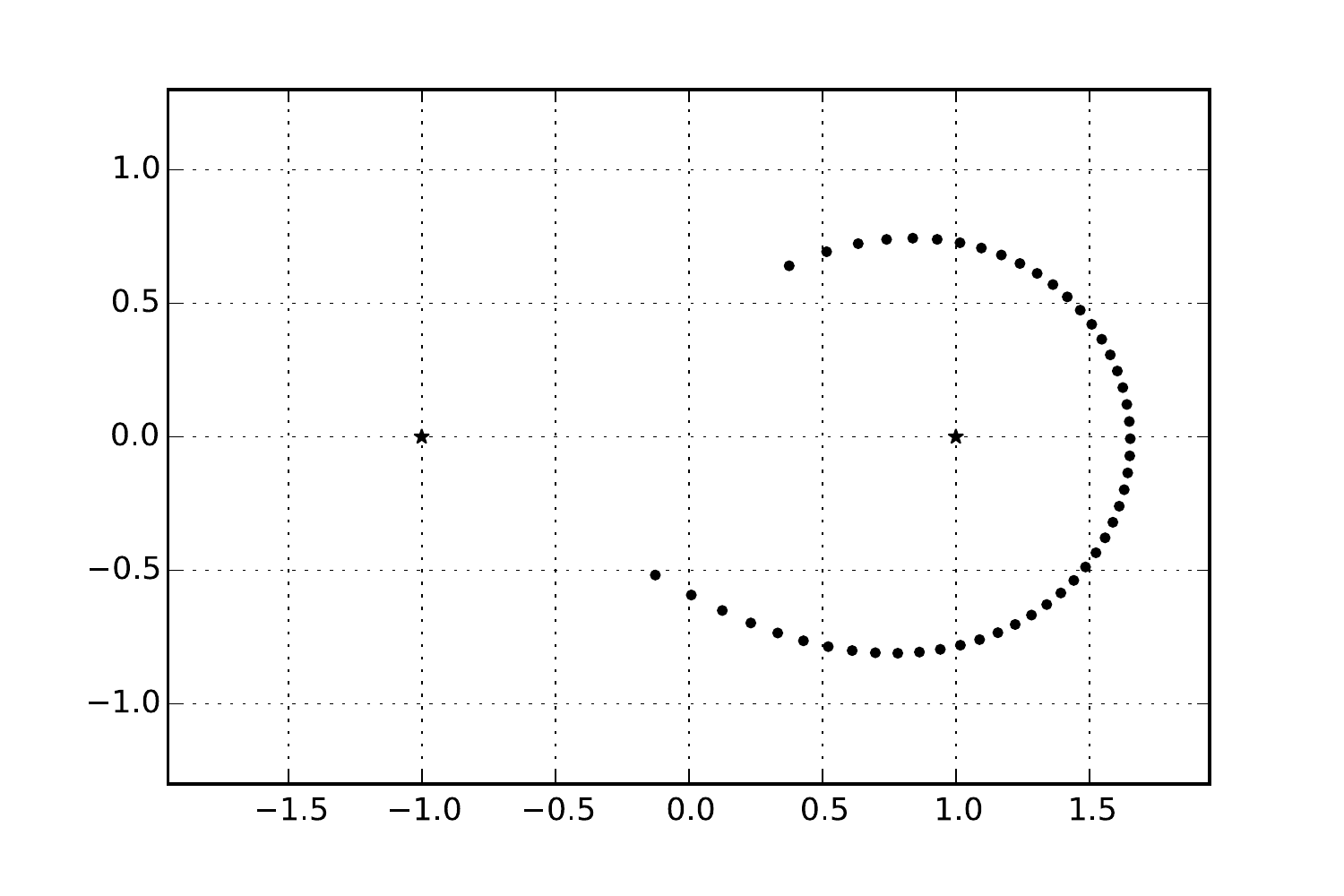}%
\put(29,31){\small $-1 $}
\put(68,33){\small $1 $}
\end{overpic}
\caption{Zeros of the polynomial $p_{50}$ for    $A=  -1.1 + 0.1i$ and $ B= 1$.}
\label{fig:onlyzeros}
\end{figure}

The key feature of the Jacobi polynomials that we use in their asymptotic analysis is their orthogonality property. It is well known that when $\alpha ,\beta  
>-1$, the Jacobi polynomials are orthogonal on $[-1,1]$
with respect to the weight function $(1-x)^{\alpha }(1+x)^{\beta }$. But as it was shown in \cite{MR2149265}, for general $\alpha ,\beta  \in \C$ we can associate with  $P_{n}^{(\alpha ,\beta )}$ a complex, non-hermitian orthogonality, where the integration goes along some contour in the complex plane. This is the key, at least in theory, to the study of the limit root location, as well as of the so-called weak (or $n$-th root) asymptotics of these polynomials (via the Gonchar-Rakhmanov-Stahl (GRS) theory \cite{Gonchar:87,Stahl:86}), and of their strong uniform asymptotics on the whole plane (by means of the Riemann-Hilbert (RH) steepest descent method of Deift-Zhou \cite{MR2000g:47048}).

The GRS method, in its general form, allows to reduce the description of the cluster set of the zeros of the sequence \eqref{eq_P}--\eqref{cond sur Aet B} to the problem of finding in a given homotopic class of curves the one  with the so-called $S$-property in the associated external field $\psi$ (see the definition \eqref{simetria} in Section~\ref{section:equilibrium}), which in our case is
\begin{equation}\label{defPsiexternalintro}
\psi(z) =-  \frac{1}{2} \Re   \left( A\log(z-1) + B\log(z+1)    \right) , \quad z\in \C\setminus \gamma_{A,B}.
\end{equation}

The most general known existence theorem for the $S$-curves is contained in \cite{Rakhmanov:2012fk} under the assumption that external field is harmonic in a complement to a finite set (see also \cite{KuijlaarsSilva2015} for the polynomial external fields). This is the case of $\psi$ in \eqref{defPsiexternalintro} for $A, B\in \R$, see \cite{MR1805976} for the analysis of the weak asymptotics (via the GRS theory) and  \cite{MR2124460,MR2142296} for the strong uniform asymptotics on $\C$ (by the RH technique).

However, $\psi$ is not harmonic (it is not even single valued) in $\C$ if we assume \eqref{cond sur Aet B}, which adds a new essential feature to the problem. In this case, the explicit construction of the curve with the $S$-property is a consequence of the analysis of  the structure of 
trajectories of the following quadratic differential on the Riemann sphere $\overline \C$:
\begin{equation*}
\varpi_{A,B}=-\frac{R_{A,B}\left( z\right) }{\left( z^{2}-1\right) ^{2}}\, dz^{2},
\end{equation*}
where  
\begin{equation*}
R_{A,B}\left( z\right) =\left( A+B+2\right) ^{2}z^{2}+2\left( 
A^{2}-B^{2}\right) z+\left( A-B\right) ^{2}-4\left( A+B+1\right) .
\end{equation*}
Although the local structure of such trajectories is well known, the global topology of the so-called critical graph is usually much more difficult to analyze. Thus, one of the central results of this paper is this description, carried out in Section~\ref{section_trajectories}, Theorem \ref{thm:1}. As a result, we claim that for every pair of parameters $(A,B)\in \C^2$ satisfying \eqref{cond sur Aet B} there exists an analytic Jordan arc $\gamma_{A, B}$, homotopic in the punctured plane $\C\setminus \{-1, 1\} $ to a Jordan arc  connecting both zeros of the polynomial $R_{A,B}$ in $\C\setminus (-\infty, 1] $, and given by the equation
$$
\Re \int^{z}  \frac{\sqrt{R_{A,B}(t)}}{t^2-1}\, dt \equiv \const.
$$
 This curve is the limiting set for the zeros of the Jacobi polynomials. Namely, with each $p_n$ we associate its normalized
zero-counting measure $\nu_n=\nu(p_n)$, such that for any compact
set $K$ in $\C$, 
\begin{equation}
\label{zerocountingmeasure}
\int_K d\nu_n=\frac{\text{number of zeros of
$p_n$ in $K$}}{n}\,. 
\end{equation}
Here the zeros are counted with their
multiplicities.

In Section~\ref{section:Jacobi} we show that the sequence $\nu_n$ converges (as $n\to \infty$) in the weak-* topology to a measure  $\mu$, supported on  $\gamma_{A,B}$, absolutely continuous with respect to
the linear Lebesgue measure on $\gamma_{A,B}$, and given by the formula
$$
\frac{d\mu(z)}{ds}=\frac{1}{2\pi }\left| \frac{\sqrt{ R_{A,B}(z)}}{z^2-1}\right|,
$$
see Theorem~\ref{teoJac1}. In Section~\ref{section:equilibrium} we establish that this is the equilibrium measure on $\gamma_{A,B}$ in an external field,  characterized by the above-mentioned $S$-property \eqref{simetria}.

It is worth noticing that the fact $\nu_n \stackrel{*}{\longrightarrow} \mu$ has an alternative interpretation from the point of view of the hypergeometric differential equation corresponding to $P_{n}^{(\alpha ,\beta )}$: for each $n\in \N$, $\nu_n$ is a discrete critical measures in the external field $\psi$,  and $\mu$ is the continuous critical measure in the same field. A general convergence theorem of this kind was proved in \cite{MR2770010}. However, the reduction of the case analyzed here to the results of \cite{MR2770010} is not direct. In particular, the existence and uniqueness of continuous critical measures for external fields with complex parameters is in general an open problem.

Our final remark is that using the construction of the measure $\mu$ and the steepest descent method for the Riemann--Hilbert characterization of the Jacobi polynomials \cite{MR2124460,MR2149265} the strong asymptotic formula can be proved. For instance (see \eqref{CauchiMu} below),
$$
\widehat\mu(z)=\int_{\gamma_{A,B}} \frac{d\mu(t)}{t-z} = \frac{1}{2} \left(\frac{A}{z-1} + \frac{B}{z+1} +    \frac{\sqrt{ R_{A,B}(z)}}{1-z^2} \right), \quad z \in \C\setminus \gamma_{A,B},
$$
where we take the  holomorphic branch of the square root in $\C\setminus \gamma_{A,B}$ such that
$$
\lim_{z\to \infty }\frac{\sqrt{R_{A,B}( z)} }{z}  =A+B+2 .
$$
Then function
$$
G(z)=\exp\left(-\int^z \widehat\mu(d) dt\right)
$$
is holomorphic in the same domain. If $\zeta_\pm$ denote the two zeros of $R_{A,B}$, let
$$
a(z)=\left( \frac{z-\zeta_+}{z-\zeta_-}\right)^{1/4}, \quad a(\infty)=1.
$$
Then there is a sequence $\kappa_n$ such that
$$
p_n(z)=\kappa_n \left( a(z)+ \frac{1}{a(z)}\right) G^n(z) \left( 1+ \mathcal O\left( \frac{1}{n}\right)\right)
$$
locally uniformly in  $\C\setminus \gamma_{A,B}$. Constants $\kappa_n$ are chosen to match the leading term of $p_n$.

This result (as well as its analogues on the limiting curve $\gamma_{A,B}$ and at $\zeta_\pm$) is established following almost literally the arguments of \cite{MR2142296}, and we refer the interested reader to that paper for details.

\section{Critical points of $\varpi_{A,B}$} \label{section_cp}

A rational quadratic differential on the Riemann sphere $\overline{\C}$ is  a form $\varpi=Q(z) dz^2$, where $Q$ is a rational function of a local coordinate $z$. If $z=z(\zeta)$ is a conformal change of variables then
$$
\widetilde Q(\zeta) d\zeta^2 = Q(z(\zeta)) (dz/\zeta)^2 d\zeta^2
$$
represents $\varpi$ in the local parameter $\zeta$. The critical points of $\varpi$ are its zeros and poles; all other points of $\overline{\C}$ are called regular points. We refer the reader to  \cite{MR0096806,Pommerenke75,Strebel:84,MR1929066} for further definitions and properties of quadratic differentials. 

In this section we focus on a specific rational quadratic differential on the Riemann sphere $\overline \C$,
\begin{equation}
\varpi_{A,B}=-\frac{R_{A,B}\left( z\right) }{\left( z^{2}-1\right) ^{2}}\, dz^{2}
\label{qd}
\end{equation}
with
\begin{equation} \label{defRAB}
R_{A,B}\left( z\right) =\left( A+B+2\right) ^{2}z^{2}+2\left( 
A^{2}-B^{2}\right) z+\left( A-B\right) ^{2}-4\left( A+B+1\right) .
\end{equation}
It depends on two parameters, $A$ and $B$, for which \eqref{cond sur Aet B} holds. 
Since
\begin{equation*}
R_{A,B}\left( z\right) = \left( \frac{z-1}{2}\right)^2 R_{-A-B-2,B}\left( \frac{z+3}{z-1}\right), \quad \overline{R_{A,B}\left( \overline{z}\right)}=R_{\overline{A},B}\left( z \right),
\end{equation*}
it is sufficient to restrict our attention to the following case:
\begin{equation} 
\Im(A)>0, \quad \Re(A)>-1-B/2, \quad B>0;
\label{ABcond}
\end{equation}
for any other combination of the parameters $(A,B)$  with $A\notin \R$ and $B>0$ we can readily derive the conclusions by combining the mappings
$$
z \mapsto \overline{z}, \quad z\mapsto \frac{z+3}{z-1}.
$$

The quadratic differential \eqref{qd} has five critical points on $\overline{\C}$; three of them at $\pm 1$ and $\infty$. Since
\begin{equation}\label{localpoles2}
\begin{split}
\varpi_{A,B} &= 
\left( -\frac{4A^{2}}{\left(
z-1\right) ^{2}}+\mathcal{O}\left( \frac{1}{z-1}\right) \right)  dz^2 , \quad z\to  1, \\
\varpi_{A,B}  &= \left( -\frac{4B^{2}}{\left(
z+1\right) ^{2}}+\mathcal{O}\left( \frac{1}{z+1}\right) \right)  dz^2 ,\quad z\to -1,  \\
\varpi_{A,B} &=  \left( -\frac{(A+B+2)^{2}}{u ^{2}}+\mathcal{O}\left( \frac{1}{u^3}\right) \right)  du^2 ,\quad u\to 0, \quad z=1/u,
\end{split}
\end{equation}
under assumptions \eqref{ABcond} these are  double poles of $\varpi_{A,B}$. The other two critical points are the zeros $\zeta _{\pm }$ of $R_{A,B} $, that we describe next.

Let $\C_{\pm}=\left\{ z\in \C :\, \pm \Im \left( z\right) >0\right\} $. 
Fixed $B>0$, we denote by
\begin{equation}
\label{defDsqrt}
D(A, B)=\sqrt{\left( A+1\right) \left(
B+1\right) \left( A+B+1\right) }
\end{equation}
the branch of this function, as a function of $A$, in the cut plane $\C\setminus (-\infty, -1]$, such that  $D(A, B)>0$ for $A>1$. Equivalently, $A\mapsto D(A, B)$ is a conformal mapping of $\C_+$ onto the upper half plane with a slit:
\begin{equation}\label{mappSQRT}
D(\cdot, B):\, \C_+ \mapsto \C_+\setminus \left\{ i x \in \C: \, x \in [0, c] \right\}, \quad c=\frac{B}{2}\sqrt{B+1}>0.
\end{equation}
With this notation, the zeros of $R_{A,B} $ are   
\begin{equation}\label{defZeta}
\zeta _{\pm }=\zeta _{\pm }(A,B)=\dfrac{-A^{2}+B^{2}\pm 4D(A, B)}{\left( A+B+2\right) ^{2}},
\end{equation}
respectively. Since $R_{A,B}\left( -1\right) =4B^{2}$ and $R_{A,B}\left( 1\right)
=4A^{2},$ it is obvious that for $A$ and $B$ satisfying (\ref{ABcond}%
), $\zeta _{+}$ and $\zeta _{-}$ are simple and different from $\pm 1$. Furthermore, the following assertions hold:
\begin{lemma}
\label{beh of zeros}
Under the assumptions \eqref{ABcond},  $\zeta_-\in \C_-$ and $\zeta _{+}\notin ( -\infty ,1] \cup [3,+\infty) $.
In particular, with $x, y\in \R$,
\begin{equation}
\label{boundaryValues1}
\lim_{y\to 0+} \zeta_{\pm}(x+i y, B)\in \begin{cases}
 \C_\pm & \text{if } x\in (-1-B/2, -1), \\
 (\R)_\pm & \text{if } -1\leq x< 0, \\
 (\R)_- & \text{if }  x> 0,
\end{cases}
\end{equation}
where $ (\R)_+$ (resp., $ (\R)_-$)  denotes the boundary values of $\R$ from the upper (resp., lower) half plane.
\end{lemma}

\begin{proof}
The polynomial $R_{A,B}$ in \eqref{defRAB} can be rewritten as
\begin{equation}
R_{A,B}(x)= \left( x+1\right)
^{2}A^{2}+2\left( B+2\right) \left( x^{2}-1\right) A+B^{2}\ \left(
x-1\right) ^{2}+4\left( x^{2}-1\right) \left( B+1\right) ;
\label{lemma2}
\end{equation}
it is a quadratic polynomial in $A$, whose discriminant  is
$$
\Delta =
 -8\left( x-1\right) \left( x+1\right) ^{2}\left( B+1\right).
$$
In particular, if $x<1$, then $\Delta>0$, so that with such $x$ the identity $R_{A,B}(x)=0$ can hold only for $A\in \R$. This proves that under assumptions \eqref{ABcond} the roots of $R_{A,B}$ cannot belong to $(-\infty, 1]$. Furthermore, if $R_{A,B}$ has a real root (hence, $>1$), by \eqref{lemma2},
$$
A=-(B+2) \frac{x-1}{x+1} + i\, \frac{  \left| \sqrt{\Delta}\right|}{(x+1)^2},
$$
and the assumption $\Re(A)>-1-B/2$ implies that $x<3$.

From the results of \cite{MR1805976} (actually, it is straightforward to check) we know that function 
$$
f_-(x)=\lim_{y\to 0+} \zeta_{-}(x+i y, B)
$$
decreases monotonically from $f_-(-1)=B-1$ to  $f_-(+\infty)=-1$ as $x$ traverses from $-1$ to $ +\infty$, while 
$$
f_+(x)=\lim_{y\to 0+} \zeta_{+}(x+i y, B)
$$
increases monotonically on $(-1,0)$, and decreases monotonically on  $(0,+\infty)$. Since $\zeta_{\pm}(z, B)$ is locally conformal,  \eqref{boundaryValues1} follows from the correspondence of boundary points.

Finally, by \eqref{defZeta},
$$
A^{2}-B^{2}+ ( A+B+2) ^{2}\zeta_-=- 4D(A, B)  .
$$ 
By \eqref{mappSQRT}, the right hand side belongs $\C_-$, so that for any  pair $(A,B)$ satisfying \eqref{ABcond}, 
$$
\Im\left( A^{2}-B^{2}+ ( A+B+2) ^{2}\zeta_-\right) = \Im \left( A^{2} + ( A+B+2) ^{2}\zeta_-\right)<0.
$$
Assuming that for certain $(A, B)$ satisfying \eqref{ABcond}, the root $\zeta_-=\zeta_-(A, B)\in \R$, and thus, $\zeta_->1$, it follows that 
$$
\Im\left( A^{2}\right) + \Im\left(( A+B+2) ^{2}\right) \zeta_- <0,
$$
or equivalently,
\begin{equation}
\label{conditionOnzeta}
  \Im\left( A \right)  \left [  \Re\left( A \right)+   \Re\left( A+B+2  \right) \zeta_- \right]  <0.
\end{equation}
However, $ \Im\left( A \right)>0$ and since $\Re(A)>-1-B/2$ and $\zeta_->1$,
$$
 \Re\left( A \right)+   \Re\left( A+B+2  \right) \zeta_- > \left( 1+\frac{B}{2}\right) \left(  \zeta_--1\right)>0,
$$
which yields a contradiction with \eqref{conditionOnzeta}. This proves that for $(A, B)$ satisfying \eqref{ABcond}, $\zeta_-\notin \R$, and thus, $\zeta_-\in \C_-$.
\end{proof}

\section{Domain configuration of $\varpi_{A,B}$} \label{section_trajectories}

Recall that the \emph{horizontal trajectories} (or just trajectories) of $\varpi_{A,B}$ are the loci of
the equation 
\begin{equation*}
\Re \int^{z}  \frac{\sqrt{R_{A,B}(t)}}{t^2-1}\, dt \equiv \const,  
\end{equation*}%
while the \emph{vertical} or \emph{orthogonal} trajectories are obtained by
replacing $\Re $ by $\Im $ in the equation above. The trajectories and the orthogonal trajectories of $\varpi_{A,B}$ produce a transversal foliation of the Riemann sphere $\overline \C$.

A trajectory $\gamma$  of $\varpi_{A,B}$ starting and ending at $\zeta_\pm$ (if exists) is called \emph{finite critical} or \emph{short}; if it starts at one of the zeros $\zeta_\pm$ but tends to either pole, we call it \emph{infinite critical trajectory} of $\varpi_{A,B}$. In a slight abuse of terminology, we say that such an infinite critical trajectory, if it exists,  \emph{joins} the zero with the corresponding pole. 
Since $\varpi_{A,B}$ has only three poles, Jenkins' three pole Theorem \cite{Jenkins:1972aa} asserts
that it cannot have any recurrent trajectory.

The set of both finite and infinite critical trajectories of $\varpi_{A,B}$ together with their limit points (critical points of $\varpi_{A,B}$) is the \emph{critical graph} $\Gamma_{A,B}$ of $\varpi_{A,B}$.

According to \cite[Theorem 3.5]{MR0096806} (see also \cite[\S 10]{Strebel:84}), the complement of the closure of $\Gamma_{A,B}$ in $\overline{\C} $ consists of a finite number of domains called the \emph{domain configuration} of $\varpi_{A,B}$. Among the possible types of domains there are the so-called circle and strip domains. A \emph{circle domain} $\mathfrak C$ of $\varpi_{A,B}$ is a maximal simply connected domain swept out by regular closed trajectories of $\varpi_{A,B}$ surrounding a double pole that is the only singularity of $\varpi_{A,B}$ in $\mathfrak C$. A \emph{strip domain} or a \emph{digon} $\mathfrak S$ of $\varpi_{A,B}$ is a maximal simply connected domain swept out by regular trajectories of $\varpi_{A,B}$, each diverging to a double pole in both directions; these double poles must represent distinct boundary points of $\mathfrak S$ (see \cite{Solynin:2009aa}).

The main result  of this section is the following theorem, which describes the critical graph as well as the domain configuration of $\varpi_{A,B}$ (see Figure~\ref{fig:globalstructure}).
\begin{theorem} \label{thm:1}
Let $A \notin  \R$ and $B>0$. Then  there exists a short trajectory $\gamma_{A,B}$ of $\varpi_{A,B}$, joining $\zeta_-$ and $\zeta_+$. 
This trajectory is unique, homotopic in the punctured plane $\C\setminus \{-1, 1\} $ to a Jordan arc connecting $\zeta_\pm$ in $\C\setminus (-\infty, 1] $.

Furthermore,  the structure of the critical graph $\Gamma_{A,B}$ of $\varpi_{A,B}$ is as follows:
\begin{itemize}
\item the short trajectory $\gamma_{A,B}$ of $\varpi_{A,B}$, joining $\zeta_-$ and $\zeta_+$;
\item  the  unique finite critical trajectory $\sigma_{-}$ of $\varpi_{A,B}$ emanating from   $\zeta_-$ and forming a closed loop, encircling $-1$;
\item the  infinite critical trajectory $\sigma_+$, emanating from   $\zeta_+$ and diverging towards $1$;
\item the  infinite critical trajectory $\sigma_\infty$, emanating from   $\zeta_+$ and diverging towards $\infty$.
\end{itemize}
$\Gamma_{A,B}$ splits $\C$ into two connected domains: the bounded circle domain $\mathfrak C$  with center at $-1$, and an unbounded strip domain $\mathfrak S$, whose boundary points are $1$ and $\infty$. 
\end{theorem}
In other words, we claim that the critical graph of $\varpi_{A,B}$ is made of 2 short and 2 infinite critical trajectories. 
Recall that it is sufficient to analyze the case when $(A,B)$ satisfy assumptions \eqref{ABcond}.

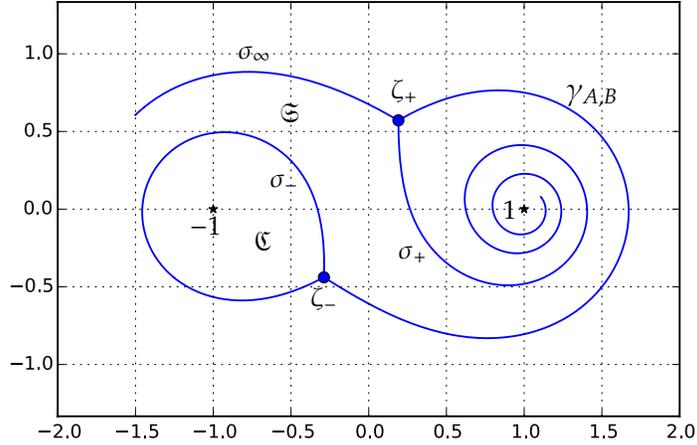
\begin{figure}[htb]
\centering \begin{overpic}[scale=0.7]{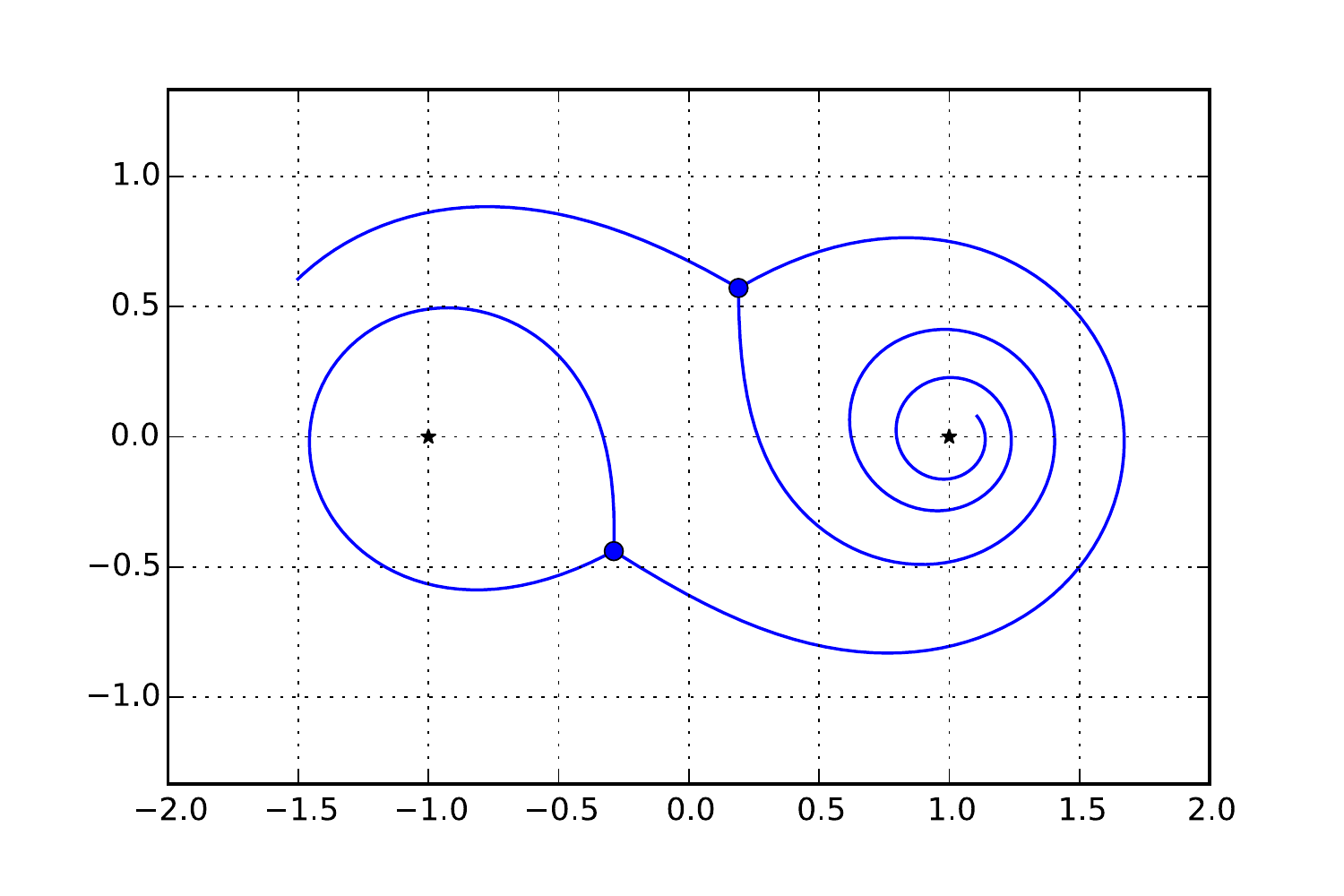}%
\put(29,31){\small $-1 $}
\put(37,29){\small $\mathfrak C$}
\put(40,45){\small $\mathfrak S$}
\put(68,33){\small $1 $}
\put(44,22){\small $\zeta_- $}
\put(54,48){\small $\zeta_+$}
\put(75,48){ $\gamma_{A,B}$}
\put(39,37){\small $\sigma_-$}
\put(55,28){\small $\sigma_+$}
\put(35,53){\small $\sigma_\infty$}
\end{overpic}
\caption{Typical structure of the critical graph $\Gamma_{A,B}$ for the trajectories   of $\varpi_{A,B}$ under the assumptions \eqref{ABcond}. These trajectories are depicted for $A=  -1.1 + 0.1i$ and $ B= 1$.}
\label{fig:globalstructure}
\end{figure}

In order to prove Theorem~\ref{thm:1} we start from the local structure of the trajectories of  $\varpi_{A,B}$ at its critical points (see e.g.~\cite{MR0096806,Pommerenke75,Strebel:84,MR1929066}). Recall that at any regular point the trajectories are locally simple analytic arcs passing through this point, and through every regular point of $\varpi_{A,B}$ passes a uniquely determined horizontal and uniquely determined vertical trajectory, mutually orthogonal at this point \cite[Theorem 5.5]{Strebel:84}. Furthermore,  there are $3$ trajectories emanating from $\zeta_\pm$ under equal angles $2\pi/3$.

By \eqref{localpoles2} we conclude that the trajectories are closed Jordan curves in a neighborhood of $-1$, and the radial or the log-spiral form at $1$ and $\infty$. The radial structure at $1$ occurs if $A\in i\R$, and at infinity, when $A+B+2\in i\R$.

Let $\gamma$ be a Jordan arc in $\C\setminus \{-1,1\}$ joining $\zeta_-$ and $\zeta_+$. Then in $\C\setminus \gamma$ we can fix a single-valued branch of $\sqrt{R_{A,B} }  $ by requiring that
\begin{equation}
\lim_{z\to \infty }\frac{\sqrt{R_{A,B}( z) }}{z} =A+B+2.  \label{asymptInfity}
\end{equation}%
Clearly, conditions
\begin{equation}
\label{eq:cond}
\sqrt{R_{A,B} ( 1) }  =   2A, \quad 
\sqrt{R_{A,B}  ( -1) }  =  -2B
\end{equation}
determine uniquely the homotopy class of $\gamma$ in the punctured plane $\C\setminus \{-1,1\}$. 
We have,
\begin{proposition}
\label{premier} Let $A, B$ satisfy assumptions \eqref{ABcond}, and let $\gamma$ be a Jordan arc in $\C\setminus \{-1,1\}$ joining $\zeta_-$ and $\zeta_+$, and $\sqrt{R_{A,B} }  $ is its single-valued branch in $\C\setminus \gamma$ fixed   by the condition \eqref{asymptInfity}. Then 
\begin{equation} \label{firstInt}
\int_{\gamma }\frac{(\sqrt{R_{A,B}\left( t\right) })_+}{t^{2}-1}dt\in \pm 2\pi
i\left\{ 1,\left( A+1\right) ,\left( B+1\right) ,\left( A+B+1\right)
\right\} ,
\end{equation}%
where $(\sqrt{R_{A,B}\left( t\right) })_+$ is the boundary value on one of the sides of $\gamma$. 

Moreover, the integral in the left hand side of \eqref{firstInt} takes the value $\pm 2\pi i$ if and only if  $\gamma $ is such that conditions \eqref{eq:cond} are satisfied.
\end{proposition}
\begin{proof}
By the properties of the square root, the integral in the left hand side of \eqref{firstInt} can be written as
$$
\frac{1}{2} \oint_{\gamma }\frac{\sqrt{R_{A,B}\left( t\right) } }{t^{2}-1}dt,
$$
which can be calculated using the residues of the integrand at $\pm 1$ and $\infty$. Thus,
\begin{align*}
\frac{1}{2} \oint_{\gamma }\frac{\sqrt{R_{A,B}\left( t\right) } }{t^{2}-1}dt & = 
 \pm  i\pi \left( \underset{-1}{\res} +\underset{1}{\res}  +\underset{\infty }{\res} \right) \left( 
\frac{\sqrt{R_{A,B} \left( t\right)}}{t^{2}-1}\right)  \\
&= \pm i\pi \left(   \frac{\sqrt{R_{A,B} ( -1)}}{-2} + \frac{\sqrt{R_{A,B} ( 1)}}{2}%
-\left( A+B+2\right) \right)  \label{poss} \\
&=  \pm 2\pi i\left\{ 1,\  A+1  , B+1  , 
A+B+1  \right\}.
\end{align*}
\end{proof}

As it will be seen in Section~\ref{section:Jacobi}, the short trajectory $\gamma$, joining the zeros $\zeta_\pm$, beings the carrier of the asymptotic zero distribution of the Jacobi polynomials, must satisfy
$$
\int_{\gamma }\frac{(\sqrt{R_{A,B}\left( t\right) })_+}{t^{2}-1}dt = \pm 2\pi i.
$$
By the proof of Proposition~\ref{premier}, this is equivalent to conditions \eqref{eq:cond}. So, we need to establish the homotopic class of curves for which conditions~\eqref{eq:cond} are satisfied. According to Proposition~\ref{lemma:lemma31} below, there cannot exist a trajectory passing through either pole $\pm 1$ and joining both zeros $\zeta_\pm$. This shows that the homotopic class of curves within the domain $(A,B)$ given by assumptions \eqref{ABcond} remains invariant, and it is sufficient to analyze the limit case $B>0$, $-1-B/2<A<-1$, for which, by Lemma \ref{beh of zeros}, $\zeta_\pm\in \C_\pm$. By \eqref{asymptInfity}--\eqref{eq:cond},
\begin{align*}
&\lim_{z\to \infty }\frac{\sqrt{R_{A,B}( z) }}{z}  =A+B+2 >0, \\
& \sqrt{R_{A,B} ( 1)}    =   2A < 0, \quad \sqrt{R_{A,B} ( -1)}    =  - 2B< 0,
\end{align*}
which shows that $\gamma$ cuts $\R$ at some point $x>1$. We conclude that
\begin{proposition}
\label{homotopy} 
Under assumptions \eqref{ABcond}, Jordan arcs $\gamma$ joining $\zeta_-$ and $\zeta_+$, and such that conditions~\eqref{eq:cond} are satisfied, are homotopic in the punctured plane $\C\setminus \{-1, 1\} $ to a Jordan arc connecting $\zeta_\pm$ in $\C\setminus (-\infty, 1] $.
\end{proposition}

Another tool needed to finish the proof of Theorem~\ref{thm:1} is the following result:
\begin{proposition} \label{lemma:lemma31}
Under assumptions \eqref{ABcond}, 
\begin{enumerate}
\item[(i)] There cannot exist two infinite critical trajectories
emanating from the zeros of $R_{A,B}$ and diverging to the pole at $z=1$.
\item[(ii)] There cannot exist two infinite critical trajectories
emanating from the same zero of $R_{A,B}$ and diverging to $\infty $.
\end{enumerate}
\end{proposition}

Its proof is based on  the so-called Teichm\"{u}ller lemma (see   \cite[Theorem 14.1]{Strebel:84}) and follows literally the arguments that have been used in  \cite[Lemma 4]{Atia:2014JMAA}. We omit repeating them here for the sake of brevity.

\medskip

Let us establish the structure of the critical graph $\Gamma_{A,B}$. Under the assumptions \eqref{ABcond}, $z=-1$ is the center of a circle domain $\mathfrak C$, whose boundary, $\partial\mathfrak{C}$, is made of critical trajectories. Since $1, \infty \notin \mathfrak C$, we conclude that $\partial \mathfrak C$ is made of short critical trajectories. Hence, a priori there are two possibilities:
\begin{enumerate}
\item[(a)] either $\partial \mathfrak{C}$ is made of two short trajectories, both connecting $\zeta_-$ and $\zeta_+$, or
\item[(b)] $\partial \mathfrak{C}$ is a single closed critical trajectory passing either through $\zeta_-$ or $\zeta_+$.
\end{enumerate}

For a fixed $B>0$ let $\overline \Omega$ be the closure of the domain defined by the conditions \eqref{ABcond} in the $A$-plane. Observe that the origin does not belong to the image of $\overline \Omega$ by the mapping \eqref{defDsqrt}--\eqref{mappSQRT}, which means that $\zeta_\pm$ are simple in the whole $\overline \Omega$. A consequence of this fact and of Proposition \ref{lemma:lemma31} is that the homotopic class in $\C\setminus\{-1,1 \}$ of the curves comprising the critical graph $\Gamma_{A,B}$ is invariant for $A\in \overline \Omega$. For $A, B>0$ the structure is well-known (see e.g.~\cite{MR1805976}): $-1<\zeta_-<\zeta_+<1$, and $\Gamma_{A,B}$ is comprised of the interval $[\zeta_-, \zeta_+]$ and of two loops, one emanating from $\zeta_-$ and encircling $-1$, and another one emanating from $\zeta_+$ and encircling $1$. In other words, it corresponds to the condition (b) above. Hence, we may discard the possibility (a) for the whole set of parameters satisfying the assumptions \eqref{ABcond}.

In the case (b), let $\zeta\in \{\zeta_-, \zeta_+\}$ be the zero of $R_{A,B}$ on the boundary of $\mathfrak{C}$. Then the third trajectory, emanating from the same zero, cannot diverge to $1$ or $\infty$: it would oblige two critical trajectories, coming from the other zero of $R_{A,B}$, to diverge to the same pole, contradicting Proposition~\ref{lemma:lemma31}.

Thus, we conclude that there exists a short trajectory, $\gamma_{A,B}$, connecting $\zeta_-$ and $\zeta_+$. Since we have discarded the case (a) mentioned above, this settles automatically the rest of the structure of the critical graph $\Gamma_{A,B}$.

Finally, the fact that it is $\zeta_-$ the zero on the boundary of $\mathfrak{C}$ (and in consequence, that $\zeta_+$ is connected with both $1$ and $\infty$ by critical trajectories) can be established by the deformation arguments, like in the proof of Proposition~\ref{homotopy}.
 
The distinguished  short trajectory $\gamma_{A,B}$ plays an essential role in what follows. For the rest of the paper we use a notation for the holomorphic branch of $\sqrt{R_{A,B} }  $ in $\C\setminus\gamma_{A,B}$:
\begin{equation}\label{branchR}
\mathcal R_{A,B}(z)=\sqrt{R_{A,B}(z) }, \quad z\in \C\setminus\gamma_{A,B}, \quad \lim_{z\to \infty }\frac{\mathcal R_{A,B}( z) }{z} =A+B+2. 
\end{equation}

Since by assumptions \eqref{ABcond}, $(A+B+2)^2\notin \R$, we have that the complement of $\Gamma_{A,B}\cup \overline{\mathfrak{C}}$ in $\C$ is a connected domain $\mathfrak S$ whose boundary points are $1$ and $\infty$ (see Figure~\ref{fig:globalstructure}). Let us show that it is actually a strip domain, as claimed.

 We introduce in $\mathfrak S$ the following analytic function,
\begin{equation} \label{def:phi}	
	\phi(z)=  \int^{z}_{\zeta_+} \frac{\mathcal R_{A,B}(t)}{t^2-1}\, dt.
\end{equation}
Let $\widehat \sigma$ be the orthogonal trajectory of $\varpi_{A,B}$ emanating from $\zeta_+$ that is the analytic continuation of the horizontal trajectory $\sigma_+$ that joins $\zeta_+$ and $1$. Function in \eqref{def:phi} is defined in such a way that
$$
\lim_{z\to \zeta_+, \, z\in \widehat{\sigma} } \phi(z)=0.
$$
\begin{proposition}
Under assumptions \eqref{ABcond}, function $\phi $ is a conformal mapping of the domain $\mathfrak S$ onto the vertical strip $0<\Re(z)< 2\pi \Im(A)$.
 \label{prop:mapping}
\end{proposition}
 \begin{proof}
We fix the orientation of the critical graph as follows: both $\sigma_\infty$ and $\sigma_+$ are emanating from $\zeta_+$,  $\gamma_{A,B}$ is entering $\zeta_+$, and $\sigma_-$ is oriented clockwise. This orientation induces the $``+''$ and $``-''$ (that is the right and left) sides of each curve, that we indicate with superscripts. For convenience, we reproduce again the Figure~\ref{fig:globalstructure} in Figure~\ref{fig:globalstructure1}, indicating now the corresponding sides of the curves.

\begin{figure}[htb]
	\centering \begin{overpic}[scale=0.7]{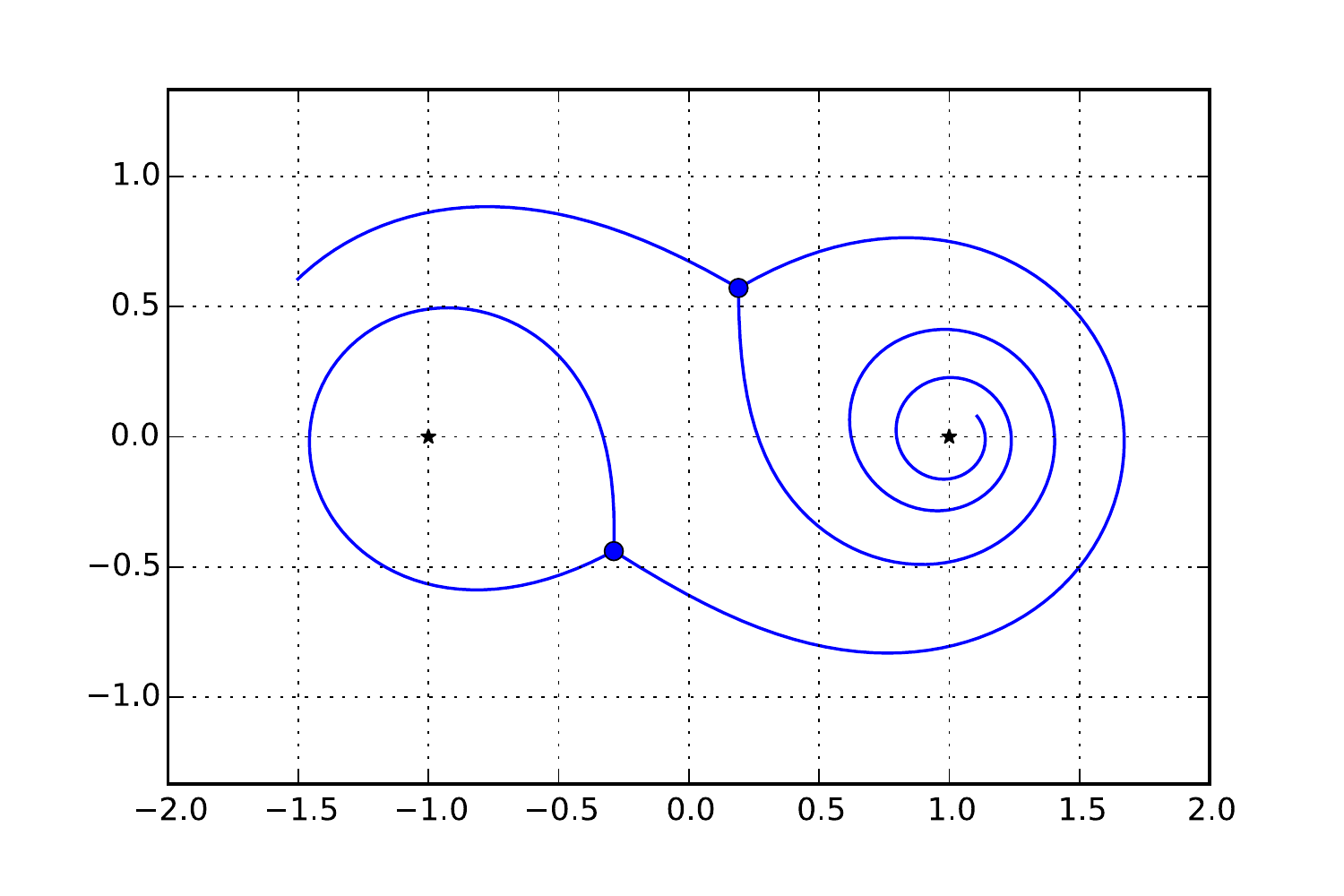}%
		\put(29,31){\small $-1 $}
		\put(37,29){\small $\mathfrak C$}
		\put(40,45){\small $\mathfrak S$}
		\put(68,33){\small $1 $}
		\put(44,22){\small $\zeta_- $}
		\put(54,48){\small $\zeta_+$}
		\put(28,48){\scriptsize $+$}
		\put(26,50){\scriptsize $-$}
		\put(81,30){\scriptsize $+$}
		\put(84,30){\scriptsize $-$}
		\put(23,26){\scriptsize $-$}
		\put(27,26){\scriptsize $+$}
		\put(53,37){\scriptsize $-$}
		\put(56,37){\scriptsize $+$}
		\put(75,48){ $\gamma_{A,B}$}
		\put(39,37){\small $\sigma_-$}
		\put(55,28){\small $\sigma_+$}
		\put(35,53){\small $\sigma_\infty$}
	\end{overpic}
	\caption{Sides of the curves forming the critical graph $\Gamma_{A,B}$ with the orientation indicated in the text.}
	\label{fig:globalstructure1}
\end{figure}

 Since $\Gamma_{A,B}$ is made of trajectories, $\phi$ maps each of these curves onto a vertical line. 
Using \eqref{eq:cond}, \eqref{branchR}, and operating as in the proof of Proposition \ref{premier}, we have 
\begin{equation} \label{2pi}
\begin{split}
\phi(\zeta_-^-)&=\lim_{z\to \zeta_-,\, z\in \gamma_{A,B}^-}\phi(z)=\int_{\zeta_{+} }^{\zeta_-}  \frac{ \mathcal R_{A,B}^-(t) }{t^2-1}\, dt \\
& = \frac{1}{2} \varointclockwise_{\gamma_{A,B} }  \frac{ \mathcal R_{A,B}(t) }{t^2-1}  \, dt
=  \pi i \left(   \frac{\mathcal R_{A,B} ( -1)}{-2} + \frac{\mathcal R_{A,B} ( 1)}{2}%
-\left( A+B+2\right) \right)   \\
&=  - 2\pi i.
\end{split}
\end{equation}
Thus, 
$$
\phi(\gamma_{A,B}^-)=(-2\pi i, 0),
$$
and in consequence, $\phi$ establishes a bijection of the boundary 
$$
\sigma_+^+ \cup \gamma_{A,B}^+ \cup \sigma_-^- \cup \gamma_{A,B}^- \cup \sigma_\infty^-
$$
of the strip domain $\mathfrak S$, oriented from 1 to $\infty$, and the imaginary axis $i\R$, oriented from $-i\infty$ to $+i\infty$. By orientation preservation, $\phi(\mathfrak S)$ lies in the right half-plane.

More precisely, let $\ell$ be a simple Jordan arc, from $\zeta_-$ to $\zeta_+$, and intersecting $\R$ only once, in $(-1,1)$. Using again the arguments from the proof of Proposition \ref{premier}, 
\begin{align*}
 \int_{\ell}  \frac{\mathcal R_{A,B}(t) }{t^2-1}\, dt & =  \frac{1}{2} \varointclockwise_{\ell }  \frac{ \mathcal R_{A,B}(t) }{t^2-1}  \, dt \\ 
& =  \pi i \left(   \frac{\mathcal R_{A,B} ( -1)}{-2} + \frac{\mathcal R_{A,B}  ( 1)}{2}%
-\left( A+B+2\right) \right)    \\
&=  - 2\pi i (A+1)= 2\pi \Im(A)- 2\pi i (\Re(A)+1) .
\end{align*}
Thus, under assumptions \eqref{ABcond}, 
$$
\Re \int_{\ell}  \frac{\mathcal R_{A,B} (t) }{t^2-1}\, dt =2\pi \Im(A)>0,
$$
which shows that the other boundary of the strip domain $\mathfrak S$ is mapped by $\phi$ onto the vertical line $\Re(z)=2\pi \Im(A)>0$. 

\end{proof}

 In the next section we will need one more technical result, related to the domain configuration of $\varpi_{A,B}$. Let $F$ be a 
Jordan curve  joining  $-1+i0$ and $-1-i0$,  lying entirely (except for its endpoints) in $\overline{\C} \setminus (-\infty, 1]$,  passing through $\zeta_\pm$ in such a way that $\gamma_{A,B}\subset F$, and otherwise disjoint with the critical graph $\Gamma_{A,B}$. We denote $F_1$ the open arc of $F$ joining $\zeta_+$ with $-1+i0$, and by $F_2$ the open arc of $F$ joining $\zeta_-$ with $-1-i0$.
 
 \begin{lemma}\label{lemma:positivity}
 	With the notations above,
 	\begin{align*}
&	\Re  \int_{\zeta_+}^z \frac{\mathcal R_{A,B}(t) }{t^2-1}\, dt < 0, \quad z\in F_1, \\
&	\Re  \int_{\zeta_-}^z \frac{\mathcal R_{A,B}(t) }{t^2-1}\, dt < 0, \quad z\in F_2. 
 	\end{align*}
 \end{lemma}
\begin{proof}
First, observe that by \eqref{eq:cond},
$$
\int_{\zeta_+}^z \frac{\mathcal R_{A,B}(t) }{t^2-1}\, dt = B\log(z+1) +\mathcal O(1), \quad z\to -1,
$$
where we choose an appropriate branch of the logarithm. This shows that the inequalities hold in a neighborhood of $z=-1$. On the other hand, assume there is a point $a\in F_1$, $a\neq \zeta_+$, such that
\begin{equation}
\label{contrac}
\Re  \int_{\zeta_+}^a \frac{\mathcal R_{A,B}(t) }{t^2-1}\, dt =0.
\end{equation}
By assumptions, $a\notin \Gamma_{A,B}$. Let $\ell$ be the horizontal trajectory of $\varpi_{A,B}$ passing through $a$; it must intersect at least one of the vertical trajectories $\ell^\perp$ of $\varpi_{A,B}$ emanating from $\zeta_+$. Hence, deforming the path from $\zeta_+$ to $a$ into the union of an arc $\ell$ and an arc from $\ell^\perp$ we run into contradiction with \eqref{contrac}.
\end{proof}

\section{An equilibrium problem for the logarithmic potential} \label{section:equilibrium}

On the short trajectory $\gamma_{A,B}$ we define the following measure, absolutely continuous with respect to the arc-length measure:
\begin{equation}
\label{def:mu}
d\mu(z)=\frac{1}{2\pi i} \frac{\mathcal R_{A,B}^+(z)}{1-z^2}\, dz,
\end{equation}
with $\mathcal R_{A,B}$ defined in \eqref{branchR}, and the $+$ boundary values are with respect to the chosen orientation of $\gamma_{A,B}$. Since $\gamma_{A,B}$ is a horizontal trajectory of $\varpi_{A,B}$, and using  \eqref{2pi} we conclude that $\mu$ is a positive probability measure defined on this arc. Straightforward calculations using residues, similar to those performed in \eqref{2pi}, show that
\begin{equation}
\label{CauchiMu}
\int_{\gamma_{A,B}} \frac{d\mu(t)}{t-z} = \frac{1}{2} \left(\frac{A}{z-1} + \frac{B}{z+1} +    \frac{\mathcal R_{A,B}(z)}{1-z^2} \right).
\end{equation}

For measure $\mu$ on $\C$,  its
logarithmic potential is defined by 
$$ 
V^\mu(z)=-\int \log|t-z| \, d\mu(t)\,. 
$$ 
By \eqref{CauchiMu}, there exists a constant $c\in \R$ such that for $z\in \C\setminus \gamma_{A,B}$,
\begin{equation}
\begin{split}
\label{PotentialMu}
V^\mu(z) & =   \frac{1}{2} \Re  \int^z \left(\frac{A}{t-1} + \frac{B}{t+1} +    \frac{\mathcal R_{A,B}(t)}{1-t^2} \right)\, dt 
\\ & = c + \frac{1}{2} \Re   \left( A\log(z-1) + B\log(z+1) +\mathcal W(z)   \right) ,
\end{split}
\end{equation}
where 
\begin{equation}
\label{defW}
\mathcal W(z)=        \int_{\zeta_-}^z    \frac{\mathcal R_{A,B}(t)}{1-t^2}  \, dt
\end{equation}
is a multivalued analytic function in $ \C\setminus \gamma_{A,B}$ with a single-valued real part. 

Let us define
$$
\psi(z) =-  \frac{1}{2} \Re   \left( A\log(z-1) + B\log(z+1)    \right) , \quad z\in \C\setminus \gamma_{A,B}.
$$
Equation \eqref{PotentialMu} can be rewritten as
\begin{equation*} 
V^\mu(z) +\psi(z)  =   c + \Re      \mathcal W(z)    ,\quad z\in \C\setminus \gamma_{A,B}.
\end{equation*}
Since $\gamma_{A,B}$ is a trajectory of $\varpi_{A,B}$, we see that
$$
V^\mu(z) +\psi(z)  =   c ,\quad z\in \C\setminus \gamma_{A,B}.
$$

Let $F$ be a 
 Jordan curve  joining  $-1+i0$ and $-1-i0$,  lying entirely (except for its endpoints) in $\overline{\C} \setminus (-\infty, 1]$,  passing through $\zeta_\pm$ in such a way that $\gamma_{A,B}\subset F$, and otherwise disjoint with the critical graph $\Gamma_{A,B}$. From Lemma \ref{lemma:positivity} we conclude that
\begin{equation}
V^\mu(z)+\psi(z)
  \begin{cases}
    = c=\const, & \text{for $z \in \supp(\mu)=\gamma_{A,B}$}, \\
    \geq c & \text{for $z \in F$}.
  \end{cases} \nonumber
\end{equation}
This property characterizes the fact that $\mu$ is actually the \emph{equilibrium measure} of $F$ in the external field $\psi$, and $c$ is the corresponding equilibrium constant (see \cite{Gonchar:87,Saff:97}). Furthermore, for $\mathcal W$ defined in \eqref{defW} the trivial identity $\mathcal W^+(z)=\mathcal W^-(z)$ on\footnote{Here we understand by $\gamma_{A,B}$ the open arc without its endpoints $\zeta_\pm$.} $\gamma_{A,B}$  yields the so-called  \emph{$S$-property} in the
external field $\psi$:  for every $\zeta \in \gamma_{A,B}$,
\begin{equation} \label{simetria}
  \frac{\partial (V^\mu+\psi)}{\partial n_-}\,(z) =
  \frac{\partial (V^\mu+\psi)}{\partial n_+}\,(z)\,,
\end{equation}
where $n_-=-n_+$ are the normals to $\gamma_{A,B}$.

\section{Relation to the asymptotics of Jacobi polynomials with varying parameters} \label{section:Jacobi}

Let us return to the Jacobi polynomials considered in Section \ref{sec:Intro}, and  consider the case of varying coefficients $\alpha$ and $\beta$ and study the asymptotic behavior of the zeros of the  sequences of polynomials $p_n$ given in \eqref{eq_P}, where the constants $A$ and $B$ satisfy the assumptions \eqref{cond sur Aet B}. As it was mentioned, it is sufficient to restrict our attention to the case \eqref{ABcond}.

 Our main goal now is to study the convergence of the
sequence $\nu_n$ of the zero counting measures \eqref{zerocountingmeasure} in the weak-$^*$ topology and, if the limit
exists, to find it explicitly.

The main result of this section is the following theorem:
\begin{theorem} \label{teoJac1}
Let the sequence of generalized Jacobi polynomials $p_n$ in
(\ref{eq_P}) be such that the pair $(A,B)$ satisfies assumptions \eqref{ABcond}. Then there is a unique measure $\mu$
such that
\begin{equation*}
\nu_n \stackrel{*}{\longrightarrow} \mu\,, \quad n \to \infty\,.
\end{equation*}
The measure $\mu$ is supported on the short trajectory $\gamma_{A,B}$, is absolutely continuous with respect to
the linear Lebesgue measure on $\gamma_{A,B}$, and   is given by the formula \eqref{def:mu}.  
\end{theorem}

\begin{figure}[htb]
\centering \begin{overpic}[scale=0.7]{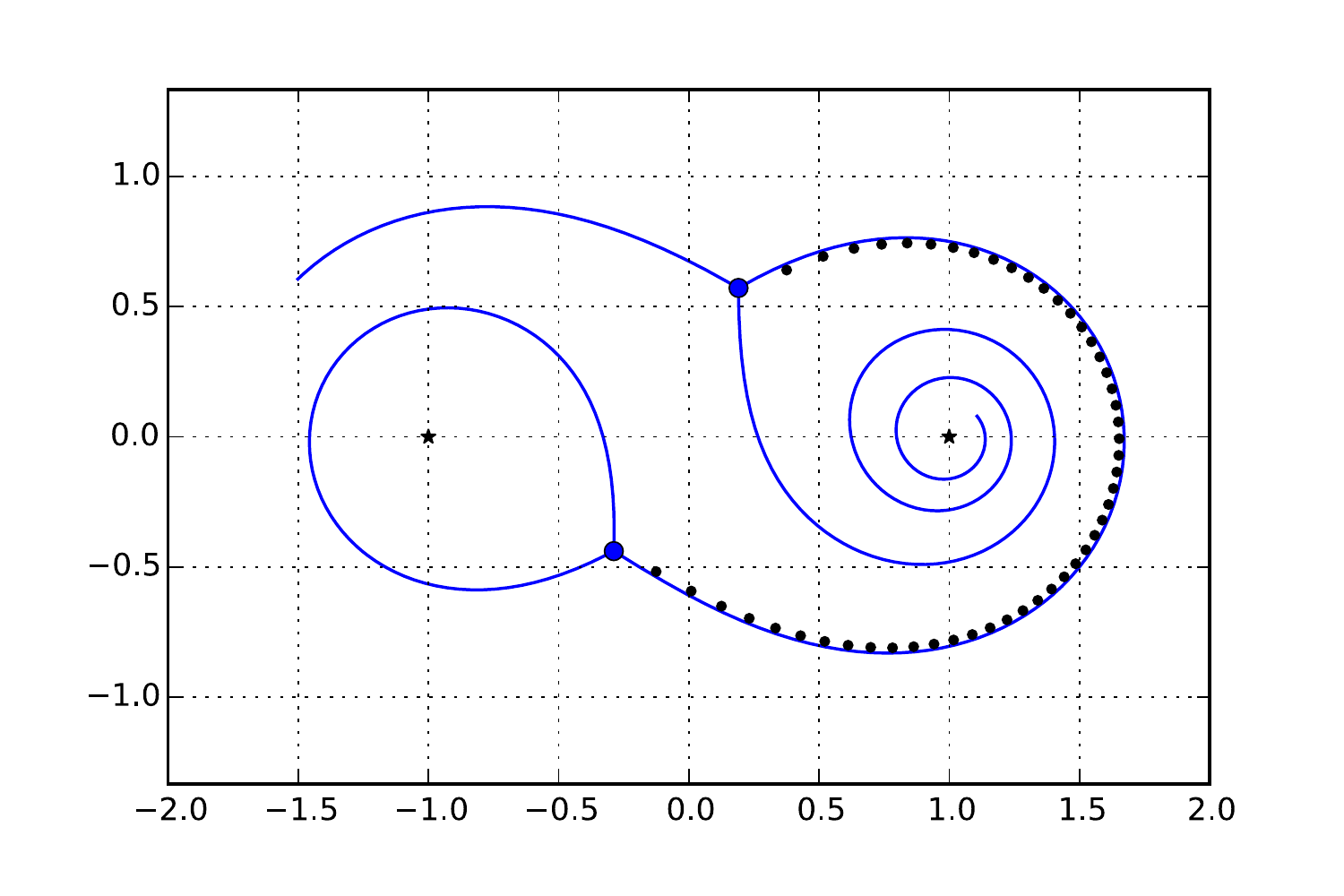}%
\put(29,31){\small $-1 $}
\put(37,29){\small $\mathfrak C$}
\put(40,45){\small $\mathfrak S$}
\put(68,33){\small $1 $}
\put(44,22){\small $\zeta_- $}
\put(54,48){\small $\zeta_+$}
\put(75,48){ $\gamma_{A,B}$}
\put(39,37){\small $\sigma_-$}
\put(55,28){\small $\sigma_+$}
\put(35,53){\small $\sigma_\infty$}
\end{overpic}
\caption{Critical graph $\Gamma_{A,B}$ of  $\varpi_{A,B}$, with $A=  -1.1 + 0.1i$ and $ B= 1$ (Figure~\ref{fig:globalstructure}), and the zeros of the corresponding polynomial $p_{50}$ (Figure~\ref{fig:onlyzeros}) superimposed.}
\label{fig:Withzeros}
\end{figure}
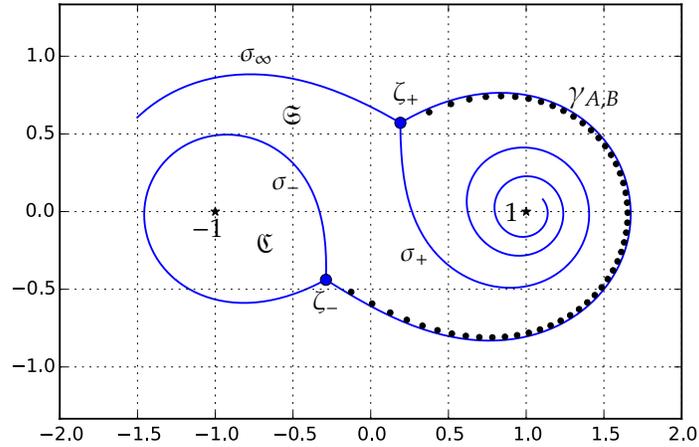

The main property satisfied by polynomials $p_n$ is the non-hermitian orthogonality conditions. 
Integrating by parts successively the Rodrigues formula
(\ref{RodrJac}), it is straightforward to obtain the following
result, proved in \cite{MR2149265}:

\begin{proposition} \label{propOrt}
Under assumptions \eqref{ABcond},  let   $F$ be a 
Jordan curve  joining  $-1+i0$ and $-1-i0$, and lying entirely (except for its endpoints) in $\overline{\C} \setminus (-\infty, 1]$. Then, for all sufficiently large $n\in \N$, 
\begin{equation} 
\oint_F P_n^{(\alpha,\beta)}(z)\, z^k (z-1)^\alpha (z+1)^\beta dz =
0\,, \quad k=0, \dots, n-1\,. \nonumber
\end{equation}
Here the integral is understood in terms of the analytic continuation
of any branch of the integrand along $F$\,.
\end{proposition}

The main tools for the study of the weak asymptotic behavior of
polynomials satisfying a non-hermitian orthogonality have been
developed in the seminal works of Stahl \cite{Stahl:86} and
Gonchar and Rakhmanov \cite{Gonchar:87}. They showed that when the
complex analytic weight function depends on the degree of the
polynomial, the limit zero distribution is characterized by an
equilibrium problem on a compact set in the presence of an
external field and satisfying the $S$-property described in Section~\ref{section:equilibrium}. In fact, Theorem~\ref{teoJac1} is a direct consequence of Proposition~\ref{propOrt}, the properties of $\mu$ established in Section~\ref{section:equilibrium}, and the original work  \cite{Gonchar:87} (see also \cite{MR1805976}). 

Finally, as it was mentioned in the Introduction, measure $\mu$ and the structure of the trajectories of $\varpi_{A,B}$ are also the main ingredients of the steepest descent method for the Riemann--Hilbert characterization of the Jacobi polynomials. The analysis follows almost literally the calculations of \cite{MR2142296}, so we refer the reader to that paper for the details.

\section*{Acknowledgments} 

The first and second authors (AMF and PMG) were partially supported by MICINN of Spain and by the
European Regional Development Fund (ERDF) under grants
MTM2011-28952-C02-01 and MTM2014-53963-P, by Junta de Andaluc\'{\i}a (the research group FQM-229), and by Campus de Excelencia Internacional del Mar (CEIMAR) of the University of Almer\'{\i}a.  
Additionally, AMF was supported by Junta de Andaluc\'{\i}a through the Excellence Grant P11-FQM-7276.
Part of this work was carried out during the visit of AMF to the Department of Mathematics of the Vanderbilt University. He acknowledges the hospitality of the hosting department, as well as a partial support of the Spanish Ministry of Education, Culture and Sports through the travel grant PRX14/00037.

We also wish to thank the anonymous referee for very useful remarks.

\bigskip 

\def\cprime{$'$}
\providecommand{\bysame}{\leavevmode\hbox to3em{\hrulefill}\thinspace}
\providecommand{\MR}{\relax\ifhmode\unskip\space\fi MR }
\providecommand{\MRhref}[2]{%
  \href{http://www.ams.org/mathscinet-getitem?mr=#1}{#2}
}
\providecommand{\href}[2]{#2}

\obeylines
\texttt{
A. Mart\'{\i}nez-Finkelshtein (andrei@ual.es)
Department of Mathematics
University of Almer\'{\i}a, Spain, and
Instituto Carlos I de F\'{\i}sica Te\'{o}rica y Computacional
Granada University, Spain
\medskip
P. Mart\'{\i}nez-Gonz\'alez (pmartine@ual.es)
Department of Mathematics
University of Almer\'{\i}a, Spain
\medskip
F. Thabet (faouzithabet@yahoo.fr)
ISSAT, University of Gabes,
Gab\'es, Tunisia
}

\end{document}